\documentclass[12pt]{article}

\usepackage[utf8]{inputenc}
\usepackage[T1]{fontenc}
\usepackage{lmodern}
\usepackage{amsfonts}
\usepackage{amssymb}
\usepackage{amsmath}
\usepackage{amsthm}
\usepackage{graphicx}
\usepackage{multirow}
\usepackage{xypic}
\usepackage{amscd}
\usepackage{mathrsfs}
\usepackage[colorlinks, linkcolor=blue, anchorcolor=black, citecolor=red]{hyperref}
\usepackage{enumerate}
\usepackage{enumitem}
\usepackage{geometry}
\usepackage{tikz-cd}
\usepackage[all]{xy}
\usepackage{cleveref}
\usepackage[autostyle,english=american]{csquotes}

\newcommand{\bC}{\mathbb{C}}

\newcommand{\bZ}{\mathbb{Z}}
\newcommand{\bR}{\mathbb{R}}

\newcommand{\bRP}{\mathbb{RP}^n}
\newcommand{\bCP}{\mathbb{CP}^n}
\newcommand{\tCP}{\mathbb{CP}^n\times\mathbb{CP}^n}

\newcommand{\bHP}{\mathbb{HP}^n}
\newcommand{\HPt}{\mathbb{HP}^2}
\newcommand{\OPt}{\mathbb{OP}^2}

\newcommand{\cO}{\mathcal{O}}

\newcommand{\qq}[1]{\enquote{#1}}

\newtheorem{thm}{Theorem}[section]
\newtheorem{prop}[thm]{Proposition}

\newtheorem{rem}[thm]{Remark}

\newtheorem{lem}[thm]{Lemma}

\newtheorem{defn-prop}[thm]{Definition-Proposition}

\title{A Note on Zoll Manifolds with Entire Grauert Tubes}
\author{Kyobeom Song}
\date{}

\begin{document}

\maketitle

\begin{abstract}
We compute several algebraic indices of the algebraization of the tangent bundle of a Zoll manifold that admits an entire Grauert tube.  As an application, we prove that any Zoll manifold of type~$\HPt$ with an entire Grauert tube is isometric to the canonical $\HPt$.
\end{abstract}

\section{Introduction}

A Riemannian manifold $(M,g)$ is called a \emph{Zoll manifold of period~$\ell$}, or simply a $C_{\ell}$-manifold, if it satisfies
\begin{enumerate}[label=(\roman*)]
  \item every geodesic $\gamma:\bR\to M$ satisfies $\gamma(t+\ell)=\gamma(t)$, and
  \item $\gamma$ is injective on $(0,\ell)$.
\end{enumerate}
A quick example is the round sphere $S^{n}$. Additional examples arise as compact rank-one symmetric spaces (CROSS) $\bRP,~\bCP,~\bHP,~\text{and}~\OPt$, with their canonical metrics.

By the Bott-Samelson theorem, any Zoll manifold has the same cohomology ring as exactly one of these five CROSSes, which we call the \emph{type} of the manifold. However, this does not imply isometry: Darboux initiated explicit investigations of $C_\ell$-manifold in 1894, and in 1903, Zoll produced the first exotic Zoll metric on $S^{2}$, which made the name. For an extensive historical reference, we recommend \cite{B12}.

\smallskip
In this paper, we particularly study Zoll manifolds whose tangent bundle admits an \emph{entire Grauert tube}: the tangent bundle $TM$ carries a complex structure~$J$ such that the function $u(v)=\|v\|$ is a plurisubharmonic exhaustion solving the homogeneous complex Monge-Ampere equation
\[
\bigl(i\partial\bar\partial u\bigr)^{n}=0,
\]
i.e.\ $(TM,J,u)$ becomes a Stein manifold with center $M$.

\paragraph{Main Question.}
Are the CROSSes the unique Zoll manifolds admitting entire Grauert tubes?

This question is important in three distinct contexts:
\begin{enumerate}
  \item \textbf{Extension of the classical Zoll metric classification problem.}
    The most intriguing feature of a Zoll metric lies in its tangent bundle~$\mathrm{TM}$.  As is well known, the tangent bundle of a Riemannian manifold carries a natural symplectic structure given by the Poincaré $2$-form.  Moreover, the geodesic flow induced by a Zoll metric preserves this form, so that finding Zoll metrics is equivalent to determining which tangent bundles admit a symplectic manifold structure for which the geodesic flow generates an $S^1$-equivariant action.  In this setting, the Main Question asks whether strengthening the symplectic condition to require a Kähler structure excludes any exotic examples beyond the CROSS.
  
  \item \textbf{Strengthening classification results for Stein manifolds by their center.}
    In \cite{PW91}, Patrizio and Wong proved that among Stein manifolds, those whose center is a CROSS are unique up to biholomorphism for each CROSS.  This stands as one of the cornerstone classification results for Stein manifolds.  The Main Question is its natural generalization: if one replaces the requirement \qq{center is a CROSS} by \qq{center is a Zoll,} does the same uniqueness still hold?
  
\item \textbf{Exploring examples of good complexifications.}
Following Totaro \cite{T03}, a \emph{good complexification} of a closed smooth manifold \(M\) is a smooth affine algebraic variety \(U\) over \(\mathbb{R}\) such that
\[
    U(\mathbb{R})\;\cong\; M 
    \quad\text{(diffeomorphism)}
    \qquad\text{and}\qquad
    U(\mathbb{R})\hookrightarrow U(\mathbb{C})
    \;\text{ is a homotopy equivalence.}
\]

A long-standing conjecture asserts that a closed real manifold is nonnegatively curved \emph{iff} it admits a good complexification.  Manifolds with good complexifications are already known to satisfy strong topological restrictions: their Euler characteristic is non-negative (Kulkarni \cite{K78}) and, more strongly, all odd Betti numbers vanish (Totaro \cite{T03}).  If the conjecture were true, such conditions would immediately settle classical problems concerning the cohomology of nonnegatively curved manifolds, e.g., the Hopf conjecture.

At present, every \emph{known} nonnegatively curved manifold does admit a good complexification, yet producing new examples of good complexification remains notably difficult.  In this context, the Main Question acquires its significance: Burns and Leung \cite{BL18} observed that entire Grauert tubes of Zoll manifolds are themselves a good complexification.  Thus, should the Main Question turn out to have a negative answer, it would be an exotic new example of a manifold with a good complexification.

\end{enumerate}

To date, the Main Question is known to hold for type \(S^n\) and \(\mathbb{RP}^n\) \cite{BL18}, and \(\mathbb{CP}^n\) \cite{LS26}, but remains completely open for \(\mathbb{HP}^n\) and \(\mathbb{OP}^2\).  In fact, the difficulty in these two cases is expected to be substantially greater for the following reasons. The framework that has succeeded in the known cases can be summarized in three steps:
\begin{enumerate}
  \item Use the Zoll structure on \(M\) to algebraize the tangent bundle \(\mathrm{TM}\) into a projective manifold \(X\).
  \item Analyze the indices of \(X\) to show that \(X\) is Fano, and then apply Fano characterization theorems to conclude that \(X\) is biholomorphic to the corresponding model case.
  \item Exploit the well-known geometry of the model case \(X\cong\overline{Q}^n\) or \(\tCP\) to prove that the totally real submanifold \(M\subset X\) is isometric to the model.
\end{enumerate}

For the first step, the remarkable work of Burns and Leung \cite{BL18} applies to Zoll manifolds of any type. Nevertheless, the second step-determining the biholomorphic type of the
projective algebraization \(X\) of \(TM\)-faces serious obstacles when the Zoll
manifold is of type \(\bHP\) or \(\OPt\).
According to Patrizio-Wong~\cite{PW91}, the algebraizations arising in each canonical model are summarized in Table~\ref{tab:PW-algebraizations}.

\begin{table}[h]
\centering
\renewcommand{\arraystretch}{1.15}
\caption{Algebraizations of \(TM\) for the five CROSS types}
\label{tab:PW-algebraizations}
\begin{tabular}{c|ccccc}
\hline
\textbf{Canonical \(M\)}
        & \(S^{n}\) & \(\mathbb{RP}^{n}\) & \(\mathbb{CP}^{n}\) & \(\mathbb{HP}^{n}\) & \(\mathbb{OP}^{2}\) \\ \hline
Algebraization of \(TM\)
        & \(\overline{Q}^{\,n}\)
        & \(\mathbb{CP}^{n}\)
        & \(\mathbb{CP}^{n}\times\mathbb{CP}^{n}\)
        & \(\mathrm{Gr}(2,2n+2)\)
        & \(E_{6}/P_{1}\) \\
\hline
\end{tabular}
\end{table}

In the first three cases (\(S^{n}\), \(\mathbb{RP}^{n}\), and \(\mathbb{CP}^{n}\))
the algebraizations either have high index or Picard number \(2\); hence powerful characterization theorems are available and play a decisive role in resolving the corresponding instances of the Main Question.
By contrast, for $Gr(2,2n+2)$ and $E_6/P_1$, virtually no general
characterization theorems exist, and the few available results rely on strong
hypotheses concerning the VMRT or the existence of dominating curve families.
Even in the third step, the geometry of $Gr(2,2n+2)$ and $E_6/P_1$ is more difficult than the previous three cases, making these two types considerably harder than the others.

In this paper, as a preliminary study toward resolving the Main Question for the still-open types \(\mathbb{HP}^n\) and \(\mathbb{OP}^2\), we first record several observations on the various algebraic indices that the projective algebraization \(X\) must satisfy.  Using these observations, we then settle the Main Question in the case of \(\mathbb{HP}^2\):

\begin{thm}\label{main}
Let \((M,g)\) be a Zoll manifold of type \(\mathbb{HP}^2\) admitting an entire Grauert tube.  Then \((M,g)\) is, up to a constant rescaling, isometric to the quaternionic projective plane \(\mathbb{HP}^2\) equipped with its canonical metric.
\end{thm}

\textbf{Acknowledgments.} The author is deeply grateful to Professor Chi Li for many valuable suggestions and discussions, and Minghao Miao for providing the precise results on the Fano classification of Mukai. The author is partially supported by NSF (Grant No. DMS-2305296).

\section{Structure of the Tangent Bundle}

Let $(M,g)$ be a Zoll manifold with period $2\pi$ whose tangent bundle admits an \emph{entire Grauert tube} structure.  
In this section we describe several geometric structures on $TM$ and introduce its projective algebraization~$X$.

\subsection{The Riemann foliation}\label{subsec:foliation}

Fix a point $p\in M$ and a unit tangent vector $v\in T_pM$.  
Let $\gamma\colon\bR\to M$ be the geodesic with $\gamma(0)=p$ and $\gamma'(0)=v$.  
Define a map
\[
    \phi_\gamma\colon\bC \;\longrightarrow\; TM,\qquad
    \phi_\gamma(\sigma+i\tau)=\tau\,\gamma'(\sigma).
\]
If $\gamma$ and $\beta$ are geodesics having distinct images, then 
$\operatorname{im}\phi_\gamma\cap\operatorname{im}\phi_\beta=\varnothing$ inside $TM\setminus M$.  
Hence the collection $\{\operatorname{im}\phi_\gamma\}_{\gamma}$ forms a foliation of $TM\setminus M$, called the \emph{Riemann foliation}. Under the Zoll condition of period $2\pi$, note that each leaf $\phi$ can in fact be viewed as an immersion of the cylinder $\phi:\mathbb{C}/2\pi\mathbb{Z}\;\to\;TM$.

\subsection{(Entire) Grauert tubes}\label{subsec:Grauert}

For \(0<r\le\infty\), set
\[
   T^rM:=\bigl\{v\in TM\mid \|v\|<r\bigr\}.
\]
A complex structure \(J\) on the smooth manifold \(T^rM\) is called an \emph{adapted complex structure} if, for every unit-speed geodesic \(\gamma\), the map \(\phi_\gamma\) introduced above restricts to a holomorphic map
\[
   \phi_\gamma:\bigl\{z\in\mathbb C\mid |\operatorname{Im}z|<r\bigr\}\longrightarrow (T^rM,J).
\]
The complex manifold \((T^rM,J)\) is then called a \emph{Grauert tube of radius \(r\)}. If such a complex structure exists on the entire tangent bundle \(T^\infty M:=TM\), we call \((TM,J)\) an \emph{entire Grauert tube}.

An entire Grauert tube carries a special Stein structure: the triple \((TM,M,u^2)\), where
\[
   u:TM\longrightarrow\mathbb R_{\ge0},\qquad u(v)=\|v\|,
\]
is a \emph{Monge--Amp\`ere model}, with the following properties:
\begin{enumerate}[label=\textbf{(\arabic*)}]
    \item \(TM\) is a Stein manifold with \(\dim_{\mathbb C}TM=\dim_{\mathbb R}M\) and center \(M=\{u=0\}\);
    \item \(u^2\) is a non-negative, strictly plurisubharmonic exhaustion of \(TM\);
    \item on \(TM\setminus M\), the function \(u\) satisfies the homogeneous complex Monge--Amp\`ere equation
    \[
       \bigl(i\,\partial\bar\partial u\bigr)^{\dim_{\mathbb R}M}=0.
    \]
\end{enumerate}
For more comprehensive accounts of Grauert tube theory, see \cite{BK77,GS91,LS91,S91}.

\subsection{The space of oriented geodesics $D$}\label{subsec:D}

Because $(M,g)$ is Zoll, each leaf $\operatorname{im}\phi_\gamma$ is biholomorphic to the infinite cylinder.  
Adding the two limit points
\[
    \infty_\gamma := \lim_{\tau\to+\infty}\phi_\gamma(\sigma+i\tau),\qquad
    0_\gamma      := \lim_{\tau\to-\infty}\phi_\gamma(\sigma+i\tau),
\]
compactifies the leaf to a Riemann sphere.  
The point $\infty_\gamma$ corresponds to the oriented geodesic $\gamma$, while $0_\gamma$ corresponds to the reversed geodesic $-\gamma$.  
Define the \emph{space of oriented geodesics}
\[
    D \;=\;\{\infty_\gamma,\,0_\gamma \mid \gamma\text{ is an oriented geodesic of }(M,g)\}.
\]
Since $D$ is the free quotient of the unit tangent bundle $UM$ by the $S^1$-action generated by the geodesic flow,  
it inherits a smooth manifold structure.

\subsection{Projective algebraization $X$ of \texorpdfstring{$TM$}{TM}}\label{subsec:alg}

By compactifying each leaf with the two limit points introduced in~\S\ref{subsec:D}, we define  
\[
    L_\gamma \;:=\; \phi_\gamma \;\sqcup\;\{\infty_\gamma,0_\gamma\},
    \qquad
    X \;:=\; TM \;\sqcup\; D .
\]
Throughout we view $X$ as the {\em projective algebraization} of~$TM$. The name ``projective algebraization'' will be explained very soon.

\subsection{An antiholomorphic involution $N_{-1}$}

Define a map
\[
    N_{-1}\colon X\longrightarrow X,\qquad
    N_{-1}(v) \;=\; -v\;\;(v\in TM), 
    \quad
    N_{-1}(\infty_\gamma)=0_\gamma,
    \quad
    N_{-1}(0_\gamma)=\infty_\gamma.
\]
Restricted to each compactified leaf $L_\gamma\cong\mathbb{CP}^1$, the map $N_{-1}$ is the complex conjugation, hence an antiholomorphic involution.

\begin{rem}
The involution $N_{-1}$ preserves virtually all structures introduced so far:  
its fixed-point set is the zero section $M\subset X$, and it stabilizes every leaf as well as the divisor~$D$. 
\end{rem}

\subsection{Parallel vector fields}\label{subsec:parallel}

Let $\{\gamma_t\}$ be a smooth one-parameter family of geodesics with $\gamma_0=\gamma$ and $V:=\partial_t\gamma_t|_{t=0}$ a Jacobi field along $\gamma$. Setting
\[
\zeta:=\partial_tL_{\gamma_t}|_{t=0}
\]
along $L_\gamma$, we obtain a section of $TX|_{L_\gamma}$, called the \emph{parallel vector field} induced by $V$.

We record a standard choice of normal parallel vector fields along $L_\gamma$. Let $\gamma$ be a unit-speed geodesic with $\gamma(0)=p$ and $\gamma'(0)=v$, and set $m:=\dim_{\mathbb R}M$. Choose normal Jacobi fields $V_1,\ldots,V_{m-1}$ along $\gamma$ such that $V_i(0)=0$ and $\{v,V_1'(0),\ldots,V_{m-1}'(0)\}$ is an orthonormal basis of $T_pM$, and let $\zeta_1,\ldots,\zeta_{m-1}$ be the parallel vector fields induced by them.

Real parallel vector fields are defined along the leaves even for a general Zoll manifold. By the standard smooth parametrization of the space of oriented geodesics $D$ \cite[Section~3]{BL18}, the fields $\{\zeta_1,\ldots,\zeta_{m-1}\}$ are real-linearly independent on $L_\gamma\setminus\gamma$, including along $L_\gamma\cap D$. Once the adapted complex structure $J$ is introduced, their $(1,0)$-parts $\zeta_i^{1,0}:=(\zeta_i-iJ\zeta_i)/2$ are holomorphic sections of $TX|_{L_\gamma}$. Conversely, the complex structure of a Grauert tube can be recovered from these holomorphic extensions. We refer to \cite{LS91,S91} for further details.

\subsection{Results of Burns--Leung} 
We now recall the key algebraization results established in \cite{BL18,L14}:
\begin{prop}[Burns--Leung]\label{prop:BL}
Let $(M,g)$ be a Zoll manifold of period $2\pi$ which admits an entire Grauert tube.
\begin{enumerate}[label=\textbf{(\roman*)}]
    \item The complex structure on $TM$ extends smoothly to~$X$, making $X$ a compact complex manifold with $D$ a codimension-$1$ submanifold.
    \item A Kähler potential
          \[
              \rho
              \;=\;
              \log\!\bigl(1+\cosh(2 u)\bigr)
          \]
          defines a Kähler form $\omega:=-i\partial\bar\partial\rho$ on $TM$ that extends to~$X$.
    \item For the parallel vector fields $\{\zeta_1,\ldots,\zeta_{m-1}\}$ introduced in Section~\ref{subsec:parallel}, one has $\omega(\zeta_i,\zeta_j)=0$ along $L_\gamma$ for all $i,j$. In particular, $\{\zeta_1^{1,0},\ldots,\zeta_{m-1}^{1,0}\}$ are complex-linearly independent away from $\gamma$.
    \item The line bundle $\cO_X(D)$ is positive; hence $D$ is ample and $X$ is a projective manifold.
    \item The map $N_{-1}$ is an antiholomorphic involution of~$X$.
\end{enumerate}
\end{prop}

\section{Observations on \texorpdfstring{$X$}{X}}

Characterizing the variety $X$ constructed above forms the second step in the three-stage framework outlined in the Introduction for resolving the Main Question, and is the program's key component. For the two unsolved types, $\bHP$ and $\OPt$, we first record several algebraic properties of $X$.

\subsection{Cohomology of \texorpdfstring{$X$}{X}}

One structural feature of~\(X\) is that it admits a decomposition
\[
    X = TM \cup \bigl(X\!\setminus\! M\bigr),
\]
to which the Mayer-Vietoris sequence applies.  Topologically
\[
    TM \simeq M, \qquad
    X\!\setminus\! M \simeq D, \qquad
    TM \cap \bigl(X\!\setminus\! M\bigr) \simeq UM,
\]
where \(UM\) is the unit tangent bundle of~\(M\).

\medskip\noindent
\textbf{Type convention.}
We say that \(M\) is of \emph{type~\(\bHP\)} (respectively, \(\OPt\))
if it has the same cohomology ring as \(\bHP\) (respectively, \(\OPt\)), namely
\[
    H^*(\bHP)=\bZ[z]/\!\langle z^{\,n+1}\rangle,
    \qquad
    H^*(\OPt)=\bZ[z]/\!\langle z^{\,3}\rangle.
\]

\begin{prop}\label{prop:cohomology}
Let \((X,UM,D)=(X_{4,n}, UM_{4,n}, D_{4,n})\) when \(M\) is of type~\(\bHP\) (so \(a=4\)),  
and \((X,UM,D)=(X_{8,2}, UM_{8,2}, D_{8,2})\) when \(M\) is of type~\(\OPt\) (so \(a=8,\,n=2\)).  
Then:
\begin{enumerate}[label=\textbf{(\arabic*)}]
    \item For the unit tangent bundle $UM$
          \[
              H^{k}\!\bigl(UM_{a,n}\bigr)\cong
              \begin{cases}
                  \bZ, & a\mid k,\; 0\le k\le an-1,\\
                  \bZ, & a\mid(k+1),\; an+1\le k\le 2an-1,\\
                  \bZ_{\,n+1}, & k=an,\\
                  0, & \text{otherwise}.
              \end{cases}
          \]
    \item For the divisor \(D\)
          \[
              H^{k}\!\bigl(D_{a,n}\bigr)\cong
              \begin{cases}
                  \bZ^{\lfloor k/a\rfloor +1}, &
                      2\mid k,\; 0\le k\le an-2,\\[2pt]
                  \bZ^{\lfloor(2an-2-k)/a\rfloor +1}, &
                      2\mid k,\; an\le k\le 2an-2,\\[2pt]
                  0, & \text{otherwise}.
              \end{cases}
          \]
    \item For the projective algebraization \(X\)
          \[
              H^{k}\!\bigl(X_{a,n}\bigr)\cong
              \begin{cases}
                  \bZ^{\lfloor k/a\rfloor +1}, &
                      2\mid k,\; 0\le k\le an,\\[2pt]
                  \bZ^{\lfloor(2an-k)/a\rfloor +1}, &
                      2\mid k,\; an+2\le k\le 2an,\\[2pt]
                  0, & \text{otherwise}.
              \end{cases}
          \]
\end{enumerate}
In particular, for $k\in\bZ$,
\[
    H^{k}\!\bigl(X_{4,n}\bigr)\cong H^{k}\!\bigl(\mathrm{Gr}(2,2n+2)\bigr),
    \qquad
    H^{k}\!\bigl(X_{8,2}\bigr)\cong H^{k}\!\bigl(E_{6}/P_{1}\bigr).
\]
\end{prop}

\begin{proof}
During this proof, we abbreviate
\[ UM_{a,n}=UM,\qquad D_{a,n}=D,\qquad X_{a,n}=X. \]

\medskip
\noindent\textbf{Step 1. Two sphere fibrations for \(UM\)}
The unit tangent bundle admits the two sphere fibrations
\[
    S^{an-1}\;\hookrightarrow\;
    UM
    \xrightarrow{\;\tau\;}
    M,
    \qquad
    S^{a-1}\;\hookrightarrow\;
    UM
    \xrightarrow{\;\pi\;}
    D.
\]
The homotopy long exact sequence for the first fibration gives
\[
    \pi_{1}\!\bigl(S^{an-1}\bigr)=0
    \;\longrightarrow\;
    \pi_{1}(UM)
    \;\longrightarrow\;
    \pi_{1}(M)=0,
\]
so \(\pi_{1}(UM)=0\). Here, \(\pi_1(M)=0\) follows from the Bott–Samelson theorem.
Applying the same argument to the second fibration yields
\[
    \pi_{1}(UM)=0
    \;\longrightarrow\;
    \pi_{1}(D)
    \;\longrightarrow\;
    \pi_{0}\!\bigl(S^{a-1}\bigr)=0,
\]
hence \(\pi_{1}(D)=0\).
Because both sphere bundles have simply connected base and total space,
the Gysin exact sequence is available in each case.

\medskip
\noindent\textbf{Step 2.  Cohomology of \(UM\).}
For the fibration \(\tau\colon UM\to M\) the Gysin sequence reads
\[
    H^{k-an}(M)\;\longrightarrow\;
    H^{k}(M)\;\longrightarrow\;
    H^{k}(UM)\;\longrightarrow\;
    H^{k-an+1}(M)\;\longrightarrow\;
    H^{k+1}(M).
\]
\begin{itemize}[leftmargin=1.8em]
    \item For \(0\le k\le an-2\), from the first four terms, we obtain \(H^{k}(UM)\cong H^{k}(M)\).
    \item For \(an+1\le k\le 2an-1\), from the last four terms, we have \(H^{k}(UM)\cong H^{k-an+1}(M)\).
\end{itemize}

For the critical degrees \(k=an-1,\,an\),
\[
    H^{an-1}(M)\;\longrightarrow\;
    H^{an-1}(UM)\;\longrightarrow\;
    H^{0}(M)
    \xrightarrow{\;\cup\,\chi\;}
    H^{an}(M)
    \;\longrightarrow\;
    H^{an}(UM)
    \;\longrightarrow\;
    H^{1}(M),
\]
where \(\chi\) is the Euler class of \(\tau\).
Because \(H^{an-1}(M)=H^{1}(M)=0\),
\(H^{an-1}(UM)\) and \(H^{an}(UM)\) are the kernel and cokernel, respectively,
of the cup-product map
\(H^{0}(M)\xrightarrow{\cup\,\chi} H^{an}(M)\).
Since \(M\) is of type \(\bHP\) (respectively, \(\OPt\)),
this map is multiplication by \(n+1\).
Therefore
\[
    H^{an-1}(UM)=0,
    \qquad
    H^{an}(UM)\cong\mathbb{Z}_{\,n+1}.
\]

\medskip
\noindent\textbf{Step 3.  Cohomology of \(D\).}
Apply the Gysin sequence to the sphere bundle
\(\pi\colon UM\to D\):
\[
\begin{aligned}
    &H^{k+1}(D)
      \;\longrightarrow\;
      H^{k+1}(UM)
      \;\longrightarrow\;
      H^{k}(D)
      \;\longrightarrow\;
      H^{k+2}(D)
      \;\longrightarrow\;
      H^{k+2}(UM)
      \;\longrightarrow\;
      H^{k+1}(D).
\end{aligned}
\]
\medskip
\noindent
We argue by induction on \(k\) using the Gysin sequence.
\begin{itemize}[leftmargin=1.8em]
    \item \textbf{Odd degrees, first range.}  
          An induction from $k=-1$ to $k=an-1$ that involves only the \emph{third, fourth, and fifth} terms of the sequence shows $H^{k}(D)=0$.
    \item \textbf{Odd degrees, second range.}  
          A similar induction from $k=2an-1$ to $k=an+1$ based on the \emph{second, third, and fourth} terms again yields $H^{k}(D)=0$.
    \item \textbf{Even degrees.}  
          For even \(k\) we perform two inductions:
          \begin{enumerate}[label=\textbullet,leftmargin=1.2em]
              \item Starting from the $H^0(D)=\bZ$ and using the \emph{last five} terms,
                    we have $H^{k+2}(D)\cong H^k(D)\oplus H^{k+2}(UM)$ for \(k\le an+a-2\).
              \item Again, starting from the $H^{2an}(D)=0$ and using the \emph{first five} terms,
                    we have $H^{k}(D)\cong H^{k+2}(D)\oplus H^{k+1}(UM)$ for \(k\ge an+a\).
          \end{enumerate}
\end{itemize}
Combining these inductions over the odd and even degrees gives precisely the cohomology groups claimed in Proposition~\ref{prop:cohomology}.

\medskip
\noindent\textbf{Step 4.  Cohomology of \(X\).}
Consider the Mayer-Vietoris sequence for the open cover
\(X=(TM)\cup(X\setminus M)\):
\[
    H^{k-1}(UM)
      \xrightarrow{\;\delta\;}
      H^{k}(X)
      \longrightarrow
      H^{k}(M)\oplus H^{k}(D)
      \xrightarrow{\;j\;}
      H^{k}(UM)
      \xrightarrow{\;\delta\;}
      H^{k+1}(X).
\]

\paragraph{(i) Degrees \(0\le k\le an\).}
The map \(j\) is surjective because
\(\tau^{*}\colon H^{k}(M)\to H^{k}(UM)\) is surjective.
Hence \(\delta=0\) and $H^k(X)$ is the kernel of the map $j$.

\paragraph{(ii) Degrees \(an+2\le k\le 2an\).}
Using the shifted Mayer-Vietoris sequence
\[
    H^{k}(M)\oplus H^{k-1}(D)
      \xrightarrow{\;j\;}
      H^{k-1}(UM)
      \longrightarrow
      H^{k}(X)
      \longrightarrow
      H^{k}(M)\oplus H^{k}(D)
      \longrightarrow
      H^{k}(UM),
\]
we distinguish two cases.

\smallskip
\emph{Odd \(k\).}
Here \(H^{k-1}(UM)=0\) and \(H^{k}(M)\oplus H^{k}(D)=0\),
so \(H^{k}(X)=0\).

\smallskip
\emph{Even \(k\).}
Now \(H^{k-1}(D)=0\), \(H^{k}(M)=0\) (since \(k\ge an+2\)),
and \(H^{k}(UM)=0\),
therefore
\(H^{k}(X)\cong H^{k-1}(UM)\oplus H^{k}(D)\).

\paragraph{(iii) Degree \(k=an+1\).}
Because \(j\) is surjective and
\(H^{an+1}(M)\oplus H^{an+1}(D)=0\),
we have \(H^{an+1}(X)=0\).

\medskip
Combining (i)-(iii) with the computations of \(H^{*}(UM)\) and \(H^{*}(D)\)
established earlier completes the proof of Proposition~\ref{prop:cohomology}.
\end{proof}

\subsection{\texorpdfstring{$X$}{X} is Fano}

For a Zoll manifold of type \(\bHP\) or \(\OPt\) we already know that
\(H^{2}(X)\cong\mathbb{Z}\).
By comparing cocycle representatives and using Morse-index data
one can show that \(X\) is a Fano variety of Picard number~\(1\).
This opens the way to applying the rich theory of Fano varieties
to our Main Question.  We state the result precisely and then prove it.

\begin{thm}\label{thm:Fano}
For \(X_{a,n}\) as in Proposition~\textup{\ref{prop:cohomology}} we have
\[
   -K_{X_{a,n}}
   \;=\;
   \cO_{X_{a,n}}\!\bigl(\tfrac{a(n+1)}{2}\,D\bigr).
\]
In particular, \(X_{a,n}\) is a Fano variety with Picard number~\(1\).
\end{thm}

\begin{proof}
Throughout the proof write \(X=X_{a,n}\).
Because \(H^{2}(X)\cong\mathbb{Z}\), let \(\alpha\in H^{2}(X;\mathbb{Z})\)
be a positive generator.  We begin by showing \(\alpha=[D]\).

\smallskip
\noindent\textbf{Step 1.  Identifying the generator of \(H^{2}(X)\).}
From Proposition~\ref{prop:cohomology} we have the fragment
\[
\xymatrix@C=4.5em@R=0.3em{
  H^{0}(UM) \ar[r]^{\cup\,\chi} & H^{2}(D) & H^{2}(X) \ar[l]_{i^{*}} \\
  1 \ar@{|->}[r]                & \chi     & \alpha   \ar@{|->}[l]
}
\]

where \(i\colon D\hookrightarrow X\) is inclusion,
and \(\chi=e(UM)=e(N_{D/X})\) is the Euler class of the sphere bundle
\(\pi\colon UM\to D\).
Because \(UM\) is (up to isomorphism) the normal bundle of \(D\) in~\(X\),
the Thom isomorphism implies
\[
   i^{*}\alpha
   =\chi
   =e(N_{D/X})
   =c_{1}\!\bigl(\cO_{X}(D)\!\mid_{D}\bigr)
   =i^{*}[D].
\]
The map \(i^{*}\colon H^{2}(X)\to H^{2}(D)\) is an isomorphism,
so \(\alpha=[D]\).
Thus for some integer \(r\)
\[
   -K_{X}=\cO_{X}(rD).
\]

\smallskip
\noindent\textbf{Step 2.  Determining \(r\).}
Fix a closed geodesic \(\gamma\) on \(M\) and let
\(L:=L_{\gamma}\cong\mathbb{CP}^{1}\subset X\) be its compactified leaf.
By the adjunction formula
\[
   K_{L}=K_{X}\!\mid_{L}+\det N_{L/X}.
\]
Because \(L\cong\mathbb{CP}^{1}\),
\(K_{L}\cong\cO_{\mathbb{CP}^{1}}(-2)\).
Since \(L\) meets \(D\) transversally in exactly two points,
\[
   K_{X}\!\mid_{L}
     =\cO_{L}(-rD)
     =\cO_{\mathbb{CP}^{1}}(-2r).
\]

\begin{lem}\label{lem:normal}
For every leaf \(L_{\gamma}\) we have
\[
   \det N_{L_{\gamma}/X}\;\cong\;\cO_{\mathbb{CP}^{1}}\!\bigl(an+a-2\bigr).
\]
\end{lem}

\begin{proof}
Recall the parallel vector fields $\{\zeta_1,\ldots,\zeta_{an-1}\}$ from Section~\ref{subsec:parallel} and Proposition~\ref{prop:BL}(iii). Thus
\[
Z:=\zeta_1^{1,0}\wedge\cdots\wedge\zeta_{an-1}^{1,0}
\]
is a nonzero holomorphic section of $\det N_{L_\gamma/X}$ whose zero loci are only contained in $\gamma$. Since $M\subset X$ is totally real, for each $p\in M$ we have the natural identification $T_pM\otimes_{\mathbb R}\mathbb C\simeq T_pX$, under which the restriction of $\zeta_i^{1,0}$ to $\gamma$ is identified with the Jacobi field $V_i$ defining $\zeta_i$.

It follows that the total vanishing order of $Z$ on $L_\gamma$ agrees with that of $V_1\wedge\cdots\wedge V_{an-1}$ along $\gamma$. Writing $V:=(V_1,\ldots,V_{an-1})$ for the corresponding Jacobi matrix, recall the standard fact that at every conjugate time $t_0$,
\[
\mathrm{ord}_{t_0}\det V=\dim\ker V(t_0).
\]
By the Bott--Samelson theorem, the Morse index of $\gamma$ over one period is $a-1$. Hence the contribution at $p$ is $an-1$, while the total contribution from the conjugate times is $a-1$. Therefore the total vanishing order is $(an-1)+(a-1)=an+a-2$.
\end{proof}

Combining the adjunction formula with Lemma~\ref{lem:normal},
\[
   \cO_{\mathbb{CP}^{1}}(-2)
     \;=\;
     \cO_{\mathbb{CP}^{1}}(-2r)
     +
     \cO_{\mathbb{CP}^{1}}(an+a-2),
\]
so \(-2=-2r+(an+a-2)\) and therefore
\(r=\tfrac{a(n+1)}{2}\).

\smallskip
\noindent\textbf{Step 3.  Fano property.}
Since \(r>0\), the anticanonical line bundle \(-K_{X}\) is ample, making
\(X\) a Fano variety.  Because \(H^{2}(X)\cong\mathbb{Z}\),
the Picard number is \(\rho(X)=1\).
\end{proof}

\subsection{Anticanonical volume of \texorpdfstring{$X$}{X}}

If a Zoll manifold is \emph{homeomorphic} (not just a type) to $\bHP$ or $\OPt$, then it is known to have the volume of the corresponding model, provided their geodesic periods agree. More precisely, Weinstein~\cite{W74} proved that for any Zoll manifold $M$ of dimension $an$ and period~$2\pi$, there is an \emph{integer} $i(M)$ which satisfies
\[
   i(M)\;:=\;\frac{\operatorname{Vol}(M)}{\operatorname{Vol}(S^{an})}
           \;=\;\frac{1}{2}\,\langle \chi^{\,an-1},[D]\rangle.
\]
We call $i(M)$ the \emph{Weinstein number}. Reznikov~\cite{R85} showed that $i(M)=i(M_{\mathrm{Model}})$ when $M$ is homeomorphic to $\bHP$ or $\OPt$.
Since $c_1(\mathcal O_X(D))$ generates $H^2(X;\mathbb Z)$ and $-K_X=\frac{a(n+1)}{2}D$, this relation forces the anticanonical volume $(-K_X)^{an}$ of $X$ to equal that of the model variety $X_{\mathrm{Model}}$ (namely $\mathrm{Gr}(2,2n+2)$ or $E_6/P_1$):

\begin{thm}\label{thm:deg}
Let $M$ be a Zoll manifold with entire Grauert tube and period $2\pi$. If it is homeomorphic to $\bHP$ or $\OPt$, then we have
\[
   \bigl(-K_{X_{a,n}}\bigr)^{an}
   \;=\;
   \bigl(-K_{X_{\mathrm{Model}}}\bigr)^{an}.
\]
\end{thm}

\begin{proof}
Set $r=\frac{a(n+1)}{2}$. Using Poincaré duality and the theorems above, we obtain:
\begin{align*}
   \bigl(-K_{X_{a,n}}\bigr)^{an}
      &= r^{an}D^{an}
       \;=\;r^{an}\int_X c_1(D)^{an} \\[2pt]
      &= r^{an}\int_X c_1(D)^{an-1}\wedge\operatorname{PD}(D)
       \;=\;r^{an}\int_D c_1(D)^{an-1} \\[2pt]
      &= r^{an}\bigl\langle \chi^{an-1},[D]\bigr\rangle
       \;=\;2r^{an}i(M)
       \;=\;2r^{an}i\!\bigl(M_{\mathrm{Model}}\bigr)
       \;=\;\bigl(-K_{X_{\mathrm{Model}}}\bigr)^{an}.
\end{align*}
\end{proof}

\begin{rem}
Since $D$ is the ample generator of $\operatorname{Pic}(X)$, the proof also computes its top self-intersection number: $D^{an}=2\,i(M)$. Hence, under the stronger assumption that $M$ is homeomorphic to $\bHP$ or $\OPt$,
\[
   D^{4n} \;=\;
   2\,i\!\bigl(\bHP\bigr)
      = \frac{2}{\,2n+1\,}\binom{4n-1}{2n-1},
   \qquad
   D^{16} \;=\;
   2\,i\!\bigl(\OPt\bigr)
      = 78,
\]
in the respective cases.
\end{rem}

\section{Main Result for the type \texorpdfstring{$\HPt$}{HP2}}\label{sec:HP2-case}

Let $M$ be a Zoll manifold of type $\HPt$ admitting an entire Grauert tube. Then $X=X_{4,2}$ is an $8$-dimensional Fano manifold of Picard number one with $-K_X=6D$. Since $D$ generates $\operatorname{Pic}(X)$, the index of $X$ is $6$, and hence its coindex is $3$. We may therefore apply Mukai's classification of Fano manifolds of coindex three \cite{M89}. More precisely, by \cite[Theorem~2.1]{W21}, a Fano manifold of dimension $m\ge6$, Picard number one, and coindex three belongs to one of the cases in Table~\ref{tab:mukai}. Here $D$ denotes the ample generator of the Picard group, and $g=\frac{1}{2}D^m+1$.

\begin{table}[ht]
\centering
\renewcommand{\arraystretch}{1.2}
\setlength{\tabcolsep}{6pt}
\begin{tabular}{c l c c}
\hline
$g=\frac{1}{2}D^m+1$ & $X$ & $m$ & $b_4(X)$ \\
\hline
2 & $X_6\subset\mathbb{P}(1^{m+1},3)$ (weighted hypersurface) & $\ge6$ & 1 \\
3 & $X_{2,4}\subset\mathbb{P}(1^{m+2},2)$ (weighted complete intersection) & $\ge6$ & 1 \\
3 & $X_4\subset\mathbb{P}^{m+1}$ (quartic hypersurface) & $\ge6$ & 1 \\
4 & $X_{2,3}\subset\mathbb{P}^{m+2}$ (complete intersection) & $\ge6$ & 1 \\
5 & $X_{2,2,2}\subset\mathbb{P}^{m+3}$ (complete intersection) & $\ge6$ & 1 \\
6 & $\operatorname{Cone}(\mathrm{Gr}(2,5))\cap Q$ (quadric section) & $6$ & 2 \\
7 & $\operatorname{OGr}_{+}(5,10)\cap\mathbb{P}^{m+5}$ (spinor linear section) & $6\le m\le10$ & 1 \\
8 & $\mathrm{Gr}(2,6)\cap\mathbb{P}^{m+6}$ (Pl\"ucker linear section) & $6\le m\le8$ & 2 \\
9 & $\mathrm{LGr}(3,6)$ (Lagrangian Grassmannian) & $6$ & 1 \\
\hline
\end{tabular}
\caption{Mukai's classification, coindex $3$, Picard number $1$ and in dimension $m\ge6$.}
\label{tab:mukai}
\end{table}

The values of $b_4$ in the table follow from the Lefschetz theorems and Schubert calculus. Since $\dim X_{4,2}=8$ and Proposition~\ref{prop:cohomology} gives $b_4(X_{4,2})=2$, the Grassmannian case is the only possibility. Its dimension is already $8$, so the linear section has codimension zero. Thus
\[
    X_{4,2}\cong\mathrm{Gr}(2,6).
\]
Thus, together with the following proposition, Theorem~\ref{main} follows.

\begin{prop}
Let \((M,g)\) be a Zoll manifold admitting an entire Grauert tube. If the projective algebraization \(X\) of the tangent bundle \(TM\) is biholomorphic to the corresponding standard projective algebraization \(X_{\mathrm{Model}}\) of a CROSS, then \((M,g)\) is isometric to the standard CROSS, up to constant rescaling.
\end{prop}

\begin{proof}
In the \(S^n\) case, this was proved by Burns and Leung in~\cite{BL18}, and the \(\mathbb RP^n\) case can be treated by the same argument. The \(\mathbb CP^n\) case was proved in the proof of the main theorem in~\cite{LS26}. Suppose now that, in the \(\mathbb HP^n\) and \(\mathbb OP^2\) cases, a biholomorphism \(F:X\to X_0\) can also be shown to identify \(D\) and \(\tau\) with the standard divisor \(D_0\) and the standard antiholomorphic involution \(\tau_0\) on \(X_0\), respectively. Then the argument following the corresponding identification in~\cite[Main theorem]{LS26} applies verbatim and yields the desired conclusion. It therefore remains to construct a biholomorphism between \(X\) and \(X_{\mathrm{Model}}\) that simultaneously identifies \(D\) and \(N_{-1}\) with their standard counterparts.

First, since \(F(D)\) is primitive, it is a hyperplane section of \(\operatorname{Gr}(2,2n+2)\) or \(E_6/P_1\), respectively. We can therefore show that \(F(D)\) and \(D_0\) are related by an automorphism of the corresponding symmetric space.
\begin{enumerate}
\item For \(\operatorname{Gr}(2,2n+2)\), consider the linear form
\[
\omega\in(\bigwedge^2\mathbb C^{2n+2})^*
=\bigwedge^2(\mathbb C^{2n+2})^*
\]
defining \(F(D)\) in the Pl\"ucker embedding. Suppose that \(\omega\) were degenerate. Since \(\omega\) is skew-symmetric, we would have \(\dim\ker\omega\ge 2\). Hence we could choose linearly independent vectors \(u,v\in\ker\omega\). At the point \(u\wedge v\), the differential of the defining equation of \(F(D)\) is
\[
(u',v')\longmapsto \omega(u',v)+\omega(u,v')=0.
\]
Thus \(F(D)\) would be singular at \(u\wedge v\), contradicting the smoothness of \(F(D)\). Hence \(\omega\) is nondegenerate. All nondegenerate skew-symmetric forms are equivalent under a change of basis.

\item For \(E_6/P_1\), the smooth hyperplane sections of \(E_6/P_1\subset\mathbb P^{26}\) form a single orbit under \(\operatorname{Aut}(E_6/P_1)\) by \cite[Propositions~4.1 and~8.1]{LM01}.
\end{enumerate}

Thus, composing with a suitable automorphism of \(X_0\), we may choose a biholomorphism \(F:X\to X_0\) identifying \(D\) with \(D_0\). For convenience, we again denote the transported antiholomorphic involution \(F\circ\tau\circ F^{-1}\) on \(X_0\) by \(\tau\).

Let \(H:=\operatorname{Aut}(X_0,D_0)\) be the group of holomorphic automorphisms of \(X_0\) preserving \(D_0\). The divisor \(D_0\), the group \(H\), and a maximal compact subgroup \(K\subset H\) are as follows:
\[
\begin{array}{c|c|c|c}
X_0 & D_0 & H & K\\ \hline
\operatorname{Gr}(2,2n+2) & \operatorname{IG}(2,2n+2) & \operatorname{PSp}(2n+2,\mathbb C) & \operatorname{Sp}(n+1)/\{\pm I\}\\
E_6/P_1 & F_4(\mathbb C)/P_4 & F_4(\mathbb C) & F_{4,c}
\end{array}
\]
Here \(\operatorname{IG}(2,2n+2)\) denotes the symplectic isotropic Grassmannian, and \(F_{4,c}\) denotes the compact real form of \(F_4(\mathbb C)\). The group \(K\) acts on the ambient complex space by extending the action of the connected isometry group of the standard manifold \(M_0=\mathbb HP^n\) or \(\mathbb OP^2\), respectively. In particular, every \(k\in K\) preserves \(M_0\), while \(\tau_0\) fixes \(M_0\) pointwise. Hence, for every \(p\in M_0\),
\[
\tau_0 k\tau_0(p)=\tau_0(k(p))=k(p),
\]
and therefore \(\tau_0\) commutes with \(k\).

By definition of \(H\), we have \(\tau\tau_0\in H\), and hence \(\tau\in H\rtimes\langle\tau_0\rangle\). Since \(K\rtimes\langle\tau_0\rangle\) is a maximal compact subgroup of \(H\rtimes\langle\tau_0\rangle\), some conjugate of \(\tau\) lies in it. Thus there exist \(h\in H\) and \(k\in K\) such that
\[
h\tau h^{-1}=k\tau_0.
\]
Since \(k\) commutes with \(\tau_0\), it follows that \(k^2=1\).

We now show that \(k=1\). This gives \(\tau_0=h\tau h^{-1}\), so replacing \(F\) by \(h\circ F\) yields the desired conclusion. Suppose that \(k\neq 1\). Since \(\mathbb HP^n=\operatorname{Sp}(n+1)/(\operatorname{Sp}(n)\times\operatorname{Sp}(1))\) and \(\mathbb OP^2=F_4/\operatorname{Spin}(9)\), the acting compact group and its isotropy subgroup have the same rank. Thus every \(k\in K\) has a fixed point \(p\in M_0\). An isometry of a symmetric space is determined by its value and differential at one point. Hence \(k^2=1\), \(k\neq 1\), and \(k(p)=p\) imply that there exists a nonzero vector \(v\in T_pM_0\subset T_p X_0\) such that \(dk_p(v)=-v\).

Recall that \(\tau_0\) is the standard involution and that \(k\) is induced by an isometry of \(M_0\). Thus \(k\tau_0\) reverses the orientation of the geodesic \(\gamma\subset M_0\) with \(\gamma'(0)=v\), and consequently preserves the corresponding leaf \(L_\gamma\) in the \emph{standard space}. Write \(d_+\) and \(d_-\) for the two points where this leaf meets \(D_0\).
\begin{enumerate}
\item By definition, \(\tau_0(d_+)=d_-\) and \(\tau_0(d_-)=d_+\).
\item Since \(k\) reverses the orientation of the geodesic, \(k(d_+)=d_-\) and \(k(d_-)=d_+\).
\end{enumerate}
Thus \(k\tau_0\) fixes both \(d_+\) and \(d_-\) on \(D_0\). However, \(k\tau_0=h\tau h^{-1}\), whereas \(\tau\) has no fixed points on \(D_0\). This is a contradiction, proving the desired conclusion.
\end{proof}

\bibliography{KSong}
\bibliographystyle{plain}

\vskip 3mm
\noindent
Department of Mathematics, Rutgers University, Piscataway, NJ 08854-8019.

\noindent
{\it Email:} ks1951@rutgers.edu

\end{document}